\documentclass[final]{siamart251216}
\usepackage{amsmath,amssymb,amsfonts,graphicx,epstopdf,bm}
\usepackage{amsopn,booktabs,multirow}
\usepackage{soul}
\usepackage{algorithm}
\usepackage[noend]{algorithmic}

\usepackage{float}
\ifpdf
  \DeclareGraphicsExtensions{.eps,.pdf,.png,.jpg}
\else
  \DeclareGraphicsExtensions{.eps}
\fi
\allowdisplaybreaks
\newtheorem{assumption}{Assumption}

\usepackage{hyperref}

\headers{Subspace curvature-scaling high-index saddle dynamics}{Yin, Zhang, and Zhang}

\title{Subspace curvature-scaling high-index saddle dynamics for accelerating ill-conditioned saddle point searches\thanks{Submitted to.
\funding{This work was supported by the National Natural Science Foundation of China (No. 12225102, T2321001, 12288101).}}}

\author{
Jianyuan Yin\thanks{School of Mathematical Sciences, 
Laboratory of Mathematics and Complex Systems, Ministry of Education, Beijing Normal University, Beijing 100875, China (\email{jyyin@bnu.edu.cn}).}
\and
Lei Zhang\thanks{School of Mathematical Sciences, Beijing International Center for Mathematical Research, Center for Machine Learning Research, Center for Quantitative Biology, Peking University, Beijing 100871, China (\email{zhangl@math.pku.edu.cn}).}
\and
Zhiyi Zhang\thanks{School of Mathematical Sciences, Peking University, Beijing 100871, China (\email{zzy2323@pku.edu.cn}).} 
}

\ifpdf
\hypersetup{
  pdftitle={Subspace curvature-scaling high-index saddle dynamics},
  pdfauthor={Yin, Zhang, and Zhang}
}
\fi

\begin{document}

\maketitle

\begin{abstract}
We propose a subspace curvature-scaling high-index saddle dynamics (SCS-HiSD) method to accelerate high-index saddle dynamics (HiSD) for locating ill-conditioned saddle points. 
The key observation is that HiSD already computes approximations of the unstable Hessian eigenvectors during iteration, which can be used at negligible additional cost to construct an inverse-Hessian approximation on the unstable subspace. 
This subspace curvature information is incorporated to adaptively scale the dynamics along each unstable direction, eliminating the dependence of the convergence rate on the smallest-magnitude negative eigenvalues and thereby substantially accelerating the convergence for ill-conditioned saddle points. 
We establish the linear stability of the continuous SCS-HiSD system and provide a local convergence analysis for the discrete iterative scheme.
This method extends naturally to address slow convergence caused by small positive eigenvalues.
Numerical experiments on benchmark problems and a liquid-crystal model demonstrate that SCS-HiSD substantially accelerates the computation of ill-conditioned saddle points, particularly in severely ill-conditioned cases.
\end{abstract}

\begin{keywords}
  saddle point, high-index saddle dynamics, ill-conditioned problem, linear stability, local convergence, solution landscape
\end{keywords}

\begin{AMS}
65K10, 65L20, 37C10
\end{AMS}

\section{Introduction}
Exploring complex energy landscapes is a fundamental problem in physics, chemistry, and materials science \cite{stillinger1995topographic, onuchic1997theory, laio2002escaping, oganov2006crystal}. 
Local minima correspond to (meta)stable states, whereas the transitions between these states are characterized by saddle points \cite{hanggi1990reaction, weinan2010transition}. 
Primary attention has been devoted to index-1 saddle points, also known as transition states, which lie on the minimum energy path connecting two adjacent local minima, and are closely related to the corresponding transition pathways.
Developing efficient algorithms for locating saddle points is essential for a comprehensive understanding of complex energy landscapes.

Existing numerical methods for locating saddle points can be classified mainly into two categories, path-finding methods and surface-walking methods.
Representative path-finding methods include the nudged elastic band method \cite{jonsson1998nudged} and the string method \cite{e2002string}, which iteratively compute a minimum energy path connecting two local minima. 
Surface-walking methods directly search for saddle points by using local gradient and Hessian information, or their approximations. 
Typical examples include gentlest ascent dynamics \cite{e2011gentlest,gu2018simplified}, the eigenvector-following method \cite{cerjan1981finding}, the activation-relaxation method \cite{cances2009some}, the iterative minimization formulation \cite{gao2015iterative}, and dimer-type methods \cite{henkelman1999dimer,zhang2012shrinking,zhang2016optimization}. 
Some surface-walking methods for saddle points have also been generalized to data-driven cases \cite{gu2022active, bello2023gentlest, georgiou2023locating, goswami2025efficient}.

Beyond transition states, high-index saddle points can provide important structural information about energy landscapes \cite{evans2003free, Wales2004energy, kastner2008phase}. 
Several surface-walking methods designed for index-1 saddle points have been extended to locate high-index saddle points \cite{quapp2014locating}. 
Recently, high-index saddle dynamics (HiSD) has been developed to locate index-$k$ saddle points by transforming an unstable saddle point into a stable equilibrium of dynamical systems \cite{yin2019high}. 
Based on HiSD, downward and upward search algorithms have been developed to search from a known saddle point to other saddle points with different indices, thereby enabling the systematic exploration of energy landscapes and construction of solution landscapes \cite{yin2020construction,zhang2026cellsystems}. 
HiSD has been successfully applied to locate multiple saddle points in various physical systems, including liquid crystals, quasicrystals, and Bose--Einstein condensates \cite{yin2022solution,yin2021transition,yinBose}.

The convergence behavior of numerical HiSD schemes has been rigorously analyzed in recent works \cite{luo2022convergence,zhang2022error,luo2024semi,luo2026convergence}. 
Specifically, for the explicit Euler scheme, the local linear convergence rate is characterized by $1 - \mathcal{O}(1/{\kappa})$, where $\kappa$ denotes the condition number of the Hessian at the target saddle point. 
Consequently, the convergence can be substantially slowed down when the target saddle point is ill-conditioned. 
This issue becomes especially pronounced when the Hessian possesses small-magnitude eigenvalues within the unstable subspace, as the dynamics along these nearly flat directions become extremely slow. 
Although many acceleration techniques, such as Barzilai--Borwein (BB) step sizes and momentum-based methods, have been incorporated into HiSD-type algorithms, their performance can still be limited in such ill-conditioned scenarios \cite{barzilai1988bb,luo2025accelerated}. 
A large $\kappa$ can lead to slow convergence and high computational cost.
Standard Newton-type methods can, in principle, overcome the conditioning issue by exploiting second-order derivatives, while their convergence region is typically restrictive for saddle-to-saddle searches, and the computational cost is often high.

Motivated by these observations, we propose a subspace curvature-scaling HiSD (SCS-HiSD) method to accelerate HiSD for locating saddle points that are ill-conditioned due to small-magnitude negative Hessian eigenvalues.
The key idea is to exploit a distinctive feature of HiSD: the iterations already maintain approximations of the unstable Hessian eigenvectors together with their associated approximate eigenvalues, which represent the local curvature along these directions. 
These quantities are reused to construct a local inverse-Hessian approximation on the unstable subspace at negligible additional computational cost. 
This approximation is then incorporated into HiSD as an adaptive scaling: the update along each unstable direction is scaled by the reciprocal of the magnitude of the corresponding approximate eigenvalue.
We provide theoretical analysis, including linear stability of the continuous system and local convergence of the discrete scheme, showing that this scaling eliminates the dependence of the convergence rate on the smallest-magnitude negative eigenvalues, thereby substantially accelerating computation.
The method extends naturally to mitigate the convergence slowdown caused by several small positive Hessian eigenvalues of the saddle point. 
Numerical experiments are conducted to demonstrate the effectiveness of SCS-HiSD, showing that the proposed method achieves substantial acceleration, particularly in locating severely ill-conditioned saddle points.

The remainder of this paper is organized as follows. 
In \Cref{sec:2}, we briefly review the HiSD method and present a detailed description of the SCS-HiSD method. 
Linear stability analysis of the continuous SCS-HiSD is presented in \Cref{sec:3}, and the local convergence analysis of the iterative scheme is established in \Cref{sec:4}.
Numerical experiments are presented in \Cref{sec:5} to demonstrate the effectiveness of the proposed method.
Conclusions and further discussions are given in \Cref{sec:6}.

\section{SCS-HiSD method}\label{sec:2}
Let $E\in \mathcal{C}^3(\mathbb{R}^d, \mathbb{R})$ be an energy function. 
In Morse theory \cite{milnormorse}, a point $x^*$ is called an index-$k$ saddle point if it is a stationary point ($\nabla E(x^*) = 0$) and the Hessian $\nabla^2 E(x^*)$ has exactly $k$ negative eigenvalues. 
We assume that the saddle point $x^*$ is nondegenerate, i.e., the Hessian $\nabla^2 E(x^*)$ has no zero eigenvalues. 
We denote by $\|\cdot\|_2$ the Euclidean norm for vectors and the induced operator norm for matrices, and by $\langle x,y \rangle = x^\top y$ the inner product.
For symmetric matrices $A$ and $B$, we write $B \preceq A$ if $A - B$ is positive semidefinite.

\subsection{Review of HiSD}
As a surface-walking method, the HiSD for an index-$k$ saddle point ($k$-HiSD) has the form \cite{yin2019high}
\begin{equation}\label{eq:hisd}
\left\{
\begin{aligned}
\frac{\mathrm{d}x}{\mathrm{d}t} &= 
-\beta \left(I - \sum_{i=1}^{k} 2v_i v_i^{\top} \right) 
\nabla E(x),\\
\frac{\mathrm{d}v_i}{\mathrm{d}t} &= 
-\gamma \left(I - v_i v_i^{\top} - \sum_{j=1}^{i-1} 2v_j v_j^{\top} \right) \nabla^{2} E(x) v_i,
\qquad 1\leqslant i\leqslant k.
\end{aligned}
\right.
\end{equation}
Here, $x\in \mathbb{R}^d$ represents the position variable, $v_1, \dots, v_k \in \mathbb{R}^d$ are $k$ directions, and $\beta, \gamma > 0$ are relaxation parameters for the $x$- and $v_i$-updates, respectively.
Starting from an initial condition satisfying $\langle {v}_i, {v}_j\rangle= \delta_{ij}$, this dynamics preserves orthonormality throughout the evolution. 
An explicit Euler discretization of \eqref{eq:hisd} is presented in \Cref{alg:hisd}. 
The notation \texttt{EigenSol} denotes a generic eigenvector solver, such as the locally optimal block preconditioned conjugate gradient (LOBPCG) method \cite{knyazev2001lobpcg}.
It takes $\{v_i^{(n)}\}_{i=1}^k$ as initial guesses and approximates the $k$ eigenpairs corresponding to the smallest $k$ eigenvalues of $\nabla^2 E(x^{(n+1)})$.
The convergence rate and error estimates of the discrete algorithm are analyzed in \cite{zhang2022error,luo2022convergence, luo2024semi}.

\begin{algorithm}[H]
\caption{HiSD for index-$k$ saddle points}
\label{alg:hisd}
\begin{algorithmic}
\REQUIRE $k \in \mathbb{N}$, $x^{(0)} \in \mathbb{R}^d$, $\{{v}_i^{(0) }\}_{i=1}^k \subset \mathbb{R}^d$ satisfying $\langle {v}_i^{(0)}, {v}_j^{(0)}\rangle= \delta_{ij}$. 
\FOR{$n = 0, 1, \dots, N-1$}
\STATE{Update position $x$: $x^{(n+1) } = x^{(n)} - 
\beta_n \left(I - \sum\limits_{i=1}^{k} 2v_i^{(n)} v_i^{(n)\top}\right) \nabla E\left (x^{(n)}\right)$.}
\STATE{Update subspace $\mathcal{V}$: $\left\{{v}_i^{(n+1)} \right\}_{i=1}^k = \mathtt{EigenSol}\left( \left\{ {v}_i^{(n)} \right\}_{i=1}^k, \nabla^2 E\left(x^{(n+1)}\right) \right)$.}
\ENDFOR
\ENSURE $x^{(N)}$
\end{algorithmic}
\end{algorithm}

To motivate the proposed acceleration, we revisit the optimization structure underlying the dynamics \eqref{eq:hisd}. 
For a nondegenerate index-$k$ saddle point $x^*$, its Hessian matrix $\nabla^2 E(x^*)$ has exactly $k$ negative eigenvalues $\lambda_1^* \leqslant \lambda_2^* \leqslant \dots \leqslant \lambda_k^*$ with the corresponding orthonormal eigenvectors $v_1^*, v_2^*, \dots, v_k^*$. 
Then $\mathcal{V}^* = \operatorname{span}\{v_1^*, v_2^*, \dots, v_k^*\}$ is the unstable subspace at $x^*$, and its orthogonal complement $\mathcal{V}^{*\perp}$ is the stable subspace at $x^*$. 
Decomposing $x^*=x_\perp^* + x_\parallel^*$ where $x_\perp^*\in \mathcal{V}^{*\perp}$ and $x_\parallel^* \in \mathcal{V}^*$, the saddle point $x^*$ is equivalently characterized as the solution $(x_\perp = x_\perp^*, x_\parallel = x_\parallel^*)$ to the minimax problem:
\begin{equation}\label{eq:optx}
    \min_{x_\perp \in \mathcal{V}^{*\perp}} 
    \max_{x_\parallel \in \mathcal{V}^*} 
    E(x_\perp + x_\parallel).
\end{equation}
For $x$ near $x^*$, the eigenvectors corresponding to the smallest $k$ eigenvalues of the Hessian $\nabla^2 E(x)$ span a subspace that approximates $\mathcal{V}^*$.
In HiSD, these eigenvectors are approximated by orthonormal vectors $v_1, \dots, v_k$ and we denote the subspace $\mathcal{V} = \operatorname{span}\{v_i\}_{i=1}^k$.
Applying gradient ascent along $\mathcal{V}$ and gradient descent along $\mathcal{V}^{\perp}$, we obtain the dynamics of $x$ as
\begin{equation}\label{eq:minmax}
\beta^{-1} \frac{\mathrm{d}x}{\mathrm{d}t} = 
\underbrace{\left(I - \sum_{i=1}^{k} v_i v_i^{\top} \right) 
\left(-\nabla E(x)\right)}_{\text{minimization over $\mathcal{V}^{\perp}$}}
+
\underbrace{\eta\left(\sum_{i=1}^{k} v_i v_i^{\top}\right) 
\nabla E(x)}_{\text{maximization over $\mathcal{V}$}}.
\end{equation}
Here $\eta > 0$ balances the relative strengths of the ascent dynamics, and setting $\eta = 1$ recovers the $x$-dynamics in \eqref{eq:hisd}.
The $v_i$-dynamics are designed to track the eigenvectors associated with the smallest $k$ Hessian eigenvalues. 
The explicit Euler scheme of \eqref{eq:minmax} is as follows:
\begin{equation}\label{eq:reform_x}
x^{(n+1)} = x^{(n)} - \beta_n \Bigg[\left( I - \sum_{i=1}^{k} v_i^{(n)} v_i^{(n)\top} \right) + \eta_n \left( - \sum_{i=1}^{k} v_i^{(n)} v_i^{(n)\top} \right)\Bigg] \nabla E\left(x^{(n)}\right),
\end{equation}
where setting $\eta_n = 1$ recovers the position update step in \Cref{alg:hisd}.

\subsection{SCS-HiSD}
When the convergence bottleneck arises from small-magnitude negative Hessian eigenvalues, a natural remedy is to accelerate the gradient ascent along $\mathcal{V}$ using local curvature information. 
A standard Newton method requires solving linear systems involving the full Hessian matrix, which is computationally expensive. 
Instead, we construct a Hessian approximation on the subspace $\mathcal{V}$ and the corresponding inverse-Hessian approximation. 
Crucially, the $v_i$-dynamics in \eqref{eq:hisd} naturally maintain approximations of the Hessian eigenvectors $v_1, \dots, v_k \in \mathcal{V}$, together with their associated approximate eigenvalues $\alpha_i = v_i^\top \nabla^2 E(x)v_i$, which capture local curvature information along these directions. 
These quantities define a Hessian approximation on the subspace $\mathcal{V}$ as $\sum_{i=1}^k \alpha_i v_i v_i^{\top}$. 
If $\alpha_i\neq0$ for $i=1, \dots, k$, the corresponding inverse-Hessian approximation on $\mathcal{V}$ is given by $\sum_{i=1}^k \alpha_i^{-1} v_i v_i^{\top}$. 
We therefore define $G_k=\sum_{i=1}^k |\alpha_i|^{-1} v_i v_i^{\top}$ as a subspace curvature-scaling operator on $\mathcal{V}$ to scale the gradient-ascent component in \eqref{eq:minmax} and thereby accelerate the maximization over $\mathcal{V}$. 
The positive definiteness of $G_k$ on $\mathcal{V}$ ensures that the scaled gradient-ascent direction remains an ascent direction.
Combining this updated $x$-dynamics with the original $v_i$-dynamics, we obtain SCS-HiSD for an index-$k$ saddle point:
\begin{equation}\label{eq:newtonhisd}
\left\{
\begin{aligned}
\frac{\mathrm{d}x}{\mathrm{d}t} &= -\beta 
\left(I - \sum_{i=1}^{k} \left(1+\frac{\eta}{|\alpha_i|}\right)v_i v_i^{\top} \right) \nabla E(x) ,\\
\frac{\mathrm{d}v_i}{\mathrm{d}t} &= -\gamma \left( I - v_i v_i^{\top} -  \sum_{j=1}^{i-1} 2v_j v_j^{\top} \right) \nabla^2E(x)v_i,\quad 1\leqslant i\leqslant k,
\end{aligned}
\right.
\end{equation}
where $\beta, \gamma > 0$ are relaxation parameters inherited from HiSD, and $\eta > 0$ controls the strength of the ascent dynamics.

The corresponding algorithm based on explicit Euler discretization is presented in \Cref{alg:qnhisd}.
For numerical stability, the denominator $|\alpha_i^{(n)}|$ should be replaced by $|\alpha_i^{(n)}|_+:=\max\{|\alpha_i^{(n)}|, \varepsilon\}$ in practical computations, where $\varepsilon > 0$ is a small positive lower bound.
A key advantage of SCS-HiSD is that it leverages the $v_i$ directions already computed during the iterations, thereby introducing negligible additional computational cost.
The effectiveness of this acceleration strategy will be demonstrated both theoretically and numerically in the following sections.

\begin{algorithm}[H]
\caption{SCS-HiSD for index-$k$ saddle points}
\label{alg:qnhisd}
\begin{algorithmic}
\REQUIRE $k \in \mathbb{N}$, $x^{(0)} \in \mathbb{R}^d$, $\{{v}_i^{(0) }\}_{i=1}^k \subset \mathbb{R}^d$ satisfying $\langle {v}_i^{(0)}, {v}_j^{(0)}\rangle= \delta_{ij}$. 
\FOR{$n = 0, 1, \dots, N-1$}
\STATE{Update position $x$: $\alpha_i^{(n)} = v_i^{(n)\top}\nabla^2 E\left(x^{(n)}\right)v_i^{(n)}$,
\\
$\quad
x^{(n+1)} = x^{(n)} - 
\beta_n \left(I - \displaystyle\sum\limits_{i=1}^{k} 
\Bigg(1+ \dfrac{\eta_n}{|\alpha_i^{(n)}|}\Bigg) 
v_i^{(n)} v_i^{(n)\top}\right) \nabla E\left (x^{(n)}\right)$.}
\STATE{Update subspace $\mathcal{V}$: $\left\{{v}_i^{(n+1)} \right\}_{i=1}^k = \mathtt{EigenSol}\left( \left\{ {v}_i^{(n)} \right\}_{i=1}^k, \nabla^2 E\left(x^{(n+1)}\right) \right)$.}
\ENDFOR
\ENSURE $x^{(N)}$
\end{algorithmic}
\end{algorithm}

\subsection{SCS-HiSD with additional directions}
The SCS-HiSD method \eqref{eq:newtonhisd} is designed to address the slow convergence caused by small-magnitude negative Hessian eigenvalues.
In more general settings, several positive eigenvalues may also be close to zero, and the SCS-HiSD method can be extended to address the corresponding convergence slowdown.
Specifically, when searching for an index-$k$ saddle point in the presence of $l$ small positive eigenvalues, the numerical scheme of SCS-HiSD is augmented by incorporating $l$ additional directions $v_{k+1}, \dots, v_{k+l}$ associated with the smallest $l$ positive eigenvalues.
In this extended framework, the energy is maximized along the first $k$ directions, $v_1, \dots, v_k$, and minimized along the remaining orthogonal directions, with $v_{k+1}, \dots, v_{k+l}$ incorporated under the same curvature-scaling principle.
Define $\sigma_{ik}=\begin{cases}
    +1, & i\leqslant k, \\ 
    -1, & i>k,
\end{cases}$
which distinguishes ascent and descent directions.
The SCS-HiSD for an index-$k$ saddle point with $l$ additional directions is given by
\begin{equation}\label{eq:extendedkl}
\left\{
\begin{aligned}
\frac{\mathrm{d}x}{\mathrm{d}t} &= -\beta 
\left(I - \sum_{i=1}^{k+l} \left(1+\frac{\sigma_{ik}\eta}{|\alpha_i|}\right)v_i v_i^{\top} \right) \nabla E(x) ,\\
\frac{\mathrm{d}v_i}{\mathrm{d}t} &= -\gamma \left( I - v_i v_i^{\top} -  \sum_{j=1}^{i-1} 2v_j v_j^{\top} \right) \nabla^2E(x)v_i,\quad 1\leqslant i\leqslant k+l.\\
\end{aligned}
\right.
\end{equation}
Accordingly, the position update in the numerical scheme is modified as 
\begin{equation}\label{eq:posi_pqn}
x^{(n+1)} = x^{(n)} -\beta_n 
\left(I - \sum_{i=1}^{k+l} 
\left(1 + \frac{\sigma_{ik}\eta_n}{|\alpha_i^{(n)}|} \right)
v_i^{(n)} v_i^{(n)\top}\right) \nabla E(x^{(n)}).
\end{equation}

In practical computations, the parameter $l$ is typically chosen as a small integer to control the additional computational cost. 
In the numerical experiments, we adopt this extended scheme when the smallest positive eigenvalue is close to zero. 
For the subsequent theoretical analysis, we focus on the original SCS-HiSD formulation \eqref{eq:newtonhisd}.

\section{Linear stability of continuous dynamics}\label{sec:3}
In this section, we clarify the relationship between the stationary points of the continuous SCS-HiSD system \eqref{eq:newtonhisd} and the saddle points of the energy landscape. 
Specifically, we demonstrate that an index-$k$ saddle point of the energy function $E$ corresponds to a linearly stable equilibrium of the dynamical system \eqref{eq:newtonhisd}, thereby establishing the validity of the proposed method in locating high-index saddle points. 

\begin{theorem}
For $E \in C^3(\mathbb{R}^d)$ and $x^* \in \mathbb{R}^d$, assume that the Hessian matrix $H^* = \nabla^2 E(x^*)$ is nondegenerate, with eigenvalues $\lambda_1^* < \dots < \lambda_k^* < \lambda_{k+1}^* \leqslant \dots \leqslant \lambda_d^*$. 
Suppose that $\{v_i^*\}_{i=1}^k \subset \mathbb{R}^d$ satisfy $\|v_i^*\| = 1$, and that $\beta, \gamma, \eta > 0$.
Then $(x^*, v_1^*, \dots, v_k^*)$ is a linearly stable equilibrium point of the dynamical system \eqref{eq:newtonhisd} if and only if $x^*$ is an index-$k$ saddle point of $E$ and each $v_i^*$ is the eigenvector of $H^*$ corresponding to the eigenvalue $\lambda_i^*$.
\end{theorem}
\begin{proof}

Consider the Jacobian $\mathbf{J}$ of \eqref{eq:newtonhisd} defined by
\begin{equation}
\mathbf{J} = \frac{\partial(\dot{x}, \dot{v}_1, \dot{v}_2, \dots, \dot{v}_k)}{\partial(x, v_1, v_2, \dots, v_k)} = \begin{pmatrix}
\mathbf{J}_x & \mathbf{J}_{x1} & \mathbf{J}_{x2} & \dots & \mathbf{J}_{xk} \\
\mathbf{J}_{1x} & \mathbf{J}_1 & \mathbf{O} & \dots & \mathbf{O} \\
\mathbf{J}_{2x} & \mathbf{J}_{21} & \mathbf{J}_2 & \dots & \mathbf{O} \\
\vdots & \vdots & \vdots & \ddots & \vdots \\
\mathbf{J}_{kx} & \mathbf{J}_{k1} & \mathbf{J}_{k2} & \dots & \mathbf{J}_k
\end{pmatrix}
\end{equation}
where the block components are explicitly given by
\begin{align}
\mathbf{J}_x = \frac{\partial \dot{x}}{\partial x} &= 
-\beta \left(I - \sum_{i=1}^k \left(1+\frac{\eta}{|\alpha_i|}\right) v_i v_i^\top \right) \nabla^2 E(x) 
+ \beta \partial_x \left( \sum_{i=1}^k \frac{\eta}{|\alpha_i|} v_i v_i^\top \right) \nabla E(x), \notag \\
\mathbf{J}_{xi} = \frac{\partial \dot{x}}{\partial v_i} &= 
\beta \left(1+\frac{\eta}{|\alpha_i|}\right)\left( v_i^\top \nabla E(x) I + v_i \nabla E(x)^\top \right) + \beta \partial_{v_i} \left( \frac{\eta}{|\alpha_i|} v_i v_i^\top \right) \nabla E(x), \notag\\
\mathbf{J}_i = \frac{\partial \dot{v}_i}{\partial v_i} &= -\gamma \left( I -  \sum_{j=1}^{i} 2 v_j v_j^\top \right) \nabla^2 E(x) 
+ \gamma  \alpha_i I , \notag
\end{align}
and $\alpha_i = v_i^\top \nabla^2 E(x) v_i$.
Note that $\mathbf{J}_{xi}$ vanishes if $\nabla E(x) = \mathbf{0}$.

``$\Leftarrow$''. Suppose that $x^*$ is an index-$k$ saddle point of $E$, and let $(\lambda_i^*, v_i^*)$ be the corresponding eigenpairs of $H^*$. 
Then, we have $\nabla E(x^*) = \mathbf{0}$ and $H^* v_i^* = \lambda_i^* v_i^*$, so $(x^*, v_1^*, \dots, v_k^*)$ is an equilibrium point of the system \eqref{eq:newtonhisd}.
By the nondegeneracy of the Hessian $H^*$, $\alpha_i$ remains negative in a neighborhood of $(x^*, v_1^*, \dots, v_k^*)$, and the system \eqref{eq:newtonhisd} is well defined and continuously differentiable.

To establish linear stability, we examine the eigenvalues of $\mathbf{J}^* = \mathbf{J}(x^*)$, which is a block lower triangular matrix due to $\nabla E(x^*) = \mathbf{0}$.
The first diagonal block,
$\mathbf{J}_x(x^*, v_1^*, \dots, v_k^*) = -\beta \Big(H^* - \sum_{i=1}^k (\lambda_i^*-\eta) v_i^* {v_i^*}^\top \Big)$,
has eigenvalues $-\beta \eta$ (multiplicity $k$), $-\beta \lambda_{k+1}^*$, $\dots$, $-\beta \lambda_d^*$. 
The other diagonal blocks,
$\mathbf{J}_i(x^*, v_1^*, \dots, v_k^*) = -\gamma \Big( H^* - \sum_{j=1}^i 2\lambda_j^* v_j^* {v_j^*}^\top - \lambda_i^* I \Big)$,
have eigenvalues $\gamma(\lambda_i^* + \lambda_1^*)$, $\dots$, $\gamma(\lambda_i^* + \lambda_i^*)$, $\gamma(\lambda_i^* - \lambda_{i+1}^*)$, $\dots$, $\gamma(\lambda_i^* - \lambda_d^*)$. Since all eigenvalues of $\mathbf{J}^*$ are negative, we conclude that $(x^*, v_1^*, \dots, v_k^*)$ is linearly stable.

``$\Rightarrow$''. Suppose $(x^*, v_1^*, \dots, v_k^*)$ is a linearly stable equilibrium point, so $\frac{\mathrm{d}v_i}{\mathrm{d}t} = \mathbf{0}$. 
Define $\mu_i^* = \langle v_i^*, H^* v_i^*\rangle$, and we prove by induction that for $i = 1, \dots, k$, the following relationships hold:
\begin{equation}\label{eq:24}
H^* v_i^* = \mu_i^* v_i^*, \quad \langle v_j^*, v_i^*\rangle = \delta_{ij}, \quad j = 1, 2, \dots, i - 1.
\end{equation}
From $\frac{\mathrm{d}v_1}{\mathrm{d}t} = \mathbf{0}$, we obtain $H^* v_1^* = \mu_1^* v_1^*$. 
Assuming that \eqref{eq:24} holds for $i \leqslant m - 1$, from $\frac{\mathrm{d}v_m}{\mathrm{d}t} = \mathbf{0}$, we obtain
\begin{equation}\label{eq:25}
\left(H^* - \sum_{j=1}^{m-1} 2\mu_j^* v_j^* {v_j^*}^\top\right) v_m^* = \mu_m^* v_m^*.
\end{equation}
Since $v_1^*, \dots, v_{m-1}^*$ are eigenvectors of $H^*$, the matrix $H^* - \sum_{j=1}^{m-1} 2\mu_j^* v_j^* {v_j^*}^\top$ shares the same eigenvectors as $H^*$, so $v_m^*$ is also an eigenvector of $H^*$ with the eigenvalue $\mu_m^*$. 
From $H^*v_m^* = \mu_m^* v_m^*$ and \eqref{eq:25}, we deduce that
$\sum_{j=1}^{m-1} \mu_j^* \langle v_j^*, v_m^* \rangle v_j^* = \mathbf{0}$. 
Since $\{v_j^*\}_{j=1}^{m-1}$ are orthogonal and $\mu_j^* \neq 0$ from nondegeneracy, it follows that
$\langle v_j^*, v_m^* \rangle = 0$ for $j < m$, which proves \eqref{eq:24}.
From $\frac{\mathrm{d}x}{\mathrm{d}t} = \mathbf{0}$, we obtain  
$\nabla E(x^*) = \mathbf{0}$, so $x^*$ is a stationary point of $E$.

In a neighborhood of $(x^*, v_1^*, \dots, v_k^*)$, $\alpha_i$ remains nonzero and the system \eqref{eq:newtonhisd} is well-defined and continuously differentiable.
We consider the eigenvalues of $\mathbf{J}^*=\mathbf{J}(x^*)$, which is a block lower triangular matrix due to $\nabla E(x^*) = \mathbf{0}$. 
Therefore, all eigenvalues of its diagonal blocks must be negative because of the linear stability.

Since $\{(\mu_j^*, v_j^*)\}_{j=1}^k$ are eigenpairs of $H^*$, we denote the other eigenvalues of $H^*$ as $\mu_{k+1}^* \leqslant \dots \leqslant \mu_d^*$.
The eigenvalues of the first diagonal block
\begin{equation}
\mathbf{J}_{x}(x^*, v_1^*, \dots, v_k^*) = -\beta 
\Big(H^* - \sum_{i=1}^k 
\Big(1+\frac{\eta}{|\mu_i^*|}\Big)
\mu_i^* v_i^* {v_i^*}^\top
\Big),
\end{equation}
are $\beta\eta \frac{\mu_1^*}{|\mu_1^*|}, \dots, \beta\eta \frac{\mu_k^*}{|\mu_k^*|}, -\beta\mu_{k+1}^*, \dots, -\beta\mu_d^*$, all of which should be negative.
Therefore, $\{\mu_i^*\}_{i=1}^k$ are negative and $\{\mu_i^*\}_{i=k+1}^d$ are positive, which implies that $x^*$ is an index-$k$ saddle point.
The eigenvalues of the diagonal block
\begin{equation}
\mathbf{J}_i(x^*, v_1^*, \dots, v_k^*) = -\gamma 
\Big(H^* - \sum_{j=1}^i 2\mu_j^* v_j^* {v_j^*}^\top -\mu_i^* I \Big),
\end{equation}
are $\gamma(\mu_i^* + \mu_1^*), \dots, \gamma(\mu_i^* + \mu_i^*), \gamma(\mu_i^* - \mu_{i+1}^*), \dots, \gamma(\mu_i^* - \mu_d^*)$, all of which should be negative. 
Therefore, we have $\mu_i^* < \mu_{i+1}^*$ for $i = 1, \dots, k$, implying that $\mu_i^*=\lambda_i^*$. 
\end{proof}

\section{Local convergence analysis}\label{sec:4}
In this section, we demonstrate that the explicit Euler scheme of the SCS-HiSD method exhibits a significantly faster convergence rate than the original HiSD method, particularly in ill-conditioned saddle-point problems. 
We aim to establish a linear convergence rate that is independent of the small-magnitude negative Hessian eigenvalues. 
We make the following assumption for the convergence analysis:
\begin{assumption}\label{ass:1}
The initial position $x^{(0)}$ lies within a neighborhood of an index-$k$ saddle point $x^*$, i.e., 
$x^{(0)} \in U(x^*, \delta) = \{ x\in\mathbb{R}^d \mid \| x - x^* \|_2 < \delta \}$,  $\delta > 0$. 
Furthermore, the following conditions hold:

(a) Lipschitz continuity: The Hessian matrix $\nabla^2 E$ is Lipschitz continuous in $U(x^*, \delta)$, that is, there exists a constant $M > 0$ such that for any $x, y \in U(x^*, \delta)$, $\| \nabla^2 E(x) - \nabla^2 E(y) \|_2 \leqslant M \| x - y \|_2$.

(b) Spectral bounds: For any $x \in U(x^*, \delta)$, the eigenvalues $\{ \lambda_i \}_{i=1}^d$ of $\nabla^2 E(x)$ satisfy $\lambda_1 \leqslant \dots \leqslant \lambda_k < 0 < \lambda_{k+1} \leqslant \dots \leqslant \lambda_d$. 
There exist positive constants $0 < \epsilon < \mu < L$ such that $|\lambda_i| \in [\epsilon, L]$ for $1 \leqslant i \leqslant k$, and $|\lambda_i| \in [\mu, L]$ for $k+1 \leqslant i \leqslant d$. 
\end{assumption}

To establish a rigorous framework for the convergence analysis, we introduce two auxiliary lemmas. 
The following lemma underlies the convergence proofs of various first-order optimization algorithms, such as gradient descent \cite{Nesterov2004introductory}. 
Essentially, it exploits the properties of contraction mappings to characterize the decay of an iterative sequence.

\begin{lemma}\cite{luo2022convergence}\label{lem:series} 
Assume that the nonnegative sequence $\{r_n\}_{n\geqslant 0}$ satisfies the following recurrence inequality: 
\begin{equation} 
    r_{n+1} \leqslant (1-q) r_n + c r_n^2, \qquad n\geqslant 0,\quad q\in(0,1) ,\ c>0. 
\end{equation} 

\noindent
(a) If $r_n < q/c$ for some $n \geqslant 0$, then $r_{n+1} < r_n < q/c$. 

\noindent
(b) If $r_0 < q/c$, then $r_{n} \leqslant 
\big(\frac{1}{1+q}\big)^{n} \frac{q r_0}{q - c r_0}$ for all $n\geqslant 0$.
\end{lemma} 

The second lemma characterizes the fundamental properties of the subspace distance metric defined by orthogonal projections. 
Since the accuracy of eigenvector approximations in numerical computations plays a critical role in determining the convergence rate of HiSD, a rigorous characterization of subspace distance is essential for the subsequent convergence analysis.

\begin{lemma}\cite[Theorem 2.5.1]{golub2013matrix}\label{lem:proj-gap}
If $W=[ W_1,\; W_2 ]$, $Z=[ Z_1,\; Z_2 ]\in \mathbb{R}^{n\times n}$ are orthogonal matrices, where $W_1, Z_1 \in \mathbb{R}^{n\times k}$, and $W_2, Z_2 \in \mathbb{R}^{n\times(n-k)}$, then 
\begin{equation*}
\left\|W_1 W_1^{\top}-Z_1 Z_1^{\top}\right\|_2
= \left\|W_1^{\top} Z_2\right\|_2
= \left\|Z_1^{\top} W_2\right\|_2.
\end{equation*}
\end{lemma}

\subsection{Cases with exact eigenvectors}
In this subsection, we assume that the exact $k$ orthonormal eigenvectors, denoted by $\{u_i^{(n)}\}_{i=1}^k$, corresponding to the smallest $k$ eigenvalues $\{\lambda_i^{(n)}\}_{i=1}^k$ of $\nabla^2 E(x^{(n)})$, are applied at the $n$-th iteration step, and the iteration scheme of SCS-HiSD is formulated as follows:
\begin{equation}\label{eq:exact_iter}
x^{(n+1)} 
=x^{(n)} - \beta_n 
\left( I - \sum_{i=1}^k\left(1+\zeta_i^{(n)}\right) u_i^{(n)} u_i^{(n)\top} \right)  \nabla E(x^{(n)}),
\end{equation}
where $\zeta_i^{(n)}=\eta_n/|\lambda_i^{(n)}| = - \eta_n/\lambda_i^{(n)}$. 
For the one-step iteration \eqref{eq:exact_iter}, we have the following theorem. 

\begin{theorem}\label{thm:exact-eig}
Under Assumption~\ref{ass:1}, if for some $n \geqslant 0$, $r_n = \|x^{(n)} - x^*\|_2 < \delta$ holds for \eqref{eq:exact_iter} with $\beta_n = {2}/{(L+\mu)}$ and $\eta_n=\mu$, then the following estimate holds:
\begin{equation}\label{eq:exact_series}
    r_{n+1} = \|x^{(n+1)} - x^*\|_2 \leqslant \left(1 - \frac{2\mu}{L+\mu} \right)  r_n + \frac{\mu M}{(L+\mu)\epsilon} r_n^2.
\end{equation}
\end{theorem}
\begin{proof}
Denote $A^{(n)} = I -  \sum_{i=1}^{k} (1+\zeta_i^{(n)}) u_i^{(n)} u_i^{(n)\top}$ and $\|A^{(n)}\|_2\leqslant \mu/\epsilon$.
Since $\nabla E(x^*) =\mathbf{0}$, we directly expand the error term as follows:
\begin{equation}
\begin{aligned}
x^{(n+1) }-x^*
&= x^{(n)}-x^* - \beta_n A^{(n)}\left(\nabla E(x^{(n)}) -\nabla E(x^*) \right)\\
&= \left(I-\beta_n A^{(n)}\int_{0}^{1}\nabla^2 E\left(x^*+t(x^{(n)}-x^*) \right) \mathrm{d}t\right)(x^{(n)}-x^*)\\
&= \left(Q^{(n)}+B^{(n)}\right)(x^{(n)}-x^*),
\end{aligned}
\end{equation}
where
\begin{equation}\label{eq:qnbn}
\begin{aligned}
Q^{(n)}&=I-\beta_n A^{(n)}\nabla^2E(x^{(n)}) ,\\
B^{(n)}&=\beta_n A^{(n)} \! \left(\nabla^2E(x^{(n)}) -\int_{0}^{1}\nabla^2E\left(x^*+t(x^{(n)}-x^*) \right) \mathrm{d}t\right).
\end{aligned}
\end{equation}
Taking norms on both sides yields the inequality $r_{n+1} \leqslant \|Q^{(n)}\|_2 r_n + \|B^{(n)}\|_2 r_n$. 

Since $x^{(n)}\in U(x^*,\delta)$, Assumption~\ref{ass:1} implies the following bound for $B^{(n)}$:
\begin{equation}
\begin{aligned}
\bigl\|B^{(n)}\bigr\|_2
&\leqslant
\beta_n \bigl\|A^{(n)}\bigr\|_2
\int_{0}^{1}\bigl\|\nabla^2 E(x^{(n)}) -
\nabla^2 E(x^*+t(x^{(n)}-x^*)) \bigr\|_2 \mathrm{d}t\\
&\leqslant 
\frac{1}{2} \beta_n M \bigl\|A^{(n)}\bigr\|_2 \bigl\|x^{(n)}-x^*\bigr\|_2
= \frac{\mu M}{(L+ \mu)\epsilon}r_n.
\end{aligned}
\end{equation}
Using the spectral decomposition $\nabla^{2} E\bigl(x^{(n)}\bigr)  = \sum_{i=1}^{d} \lambda_i^{(n)}u_i^{(n)} u_i^{(n)\top}$, we obtain
\begin{equation}
M^{(n)}=A^{(n)}\nabla^2E(x^{(n)})=
\sum_{i=1}^{k} \mu u_i^{(n)} u_i^{(n)\top} 
+\sum_{j=k+1}^{d} \lambda_j^{(n)} u_j^{(n)} u_j^{(n)\top}, 
\end{equation}
where $\lambda_j^{(n)} \in [\mu, L]$ for $j>k$.
Consequently, the eigenvalues of $M^{(n)}$ lie in the interval $[\mu, L]$, and the eigenvalues of $Q^{(n)} = I - \beta_n M^{(n)}$ lie in the interval $[1 - \beta_n L, 1 - \beta_n \mu]$.
With $\beta_n = \frac{2}{L+\mu}$, we obtain $\|Q^{(n)}\|_2 \leqslant 1 - \frac{2\mu}{L+\mu}$, which completes the proof. 
\end{proof}

Based on \Cref{thm:exact-eig}, we obtain the following conclusion regarding the linear convergence rate of SCS-HiSD in the case of exact eigenvectors. 
The linear convergence rate $1-\frac{2}{\kappa+3}$, where $\kappa=L/\mu$, is independent of $\epsilon$, the lower bound on the magnitudes of the negative eigenvalues. 

\begin{theorem}\label{thm:exact_conclusion}
Under Assumption~\ref{ass:1}, if the initial guess $x^{(0)}$ satisfies
\begin{equation*}
r_0 = \|x^{(0) } - x^*\|_2 < \min\{\delta, \hat r\}, \qquad \hat r = {2\epsilon}/{M},
\end{equation*}
and the parameters are set as $\beta_n = \frac{2}{L+\mu}$, $\eta_n=\mu$ for all $n\geqslant 0$, then the sequence $x^{(n)}$ defined by \eqref{eq:exact_iter} converges to $x^*$ with a linear convergence rate:
\begin{equation}
r_n = \|x^{(n)} - x^*\|_2 \leqslant \left(1 - \frac{2}{\kappa + 3}\right) ^n \frac{\hat r  r_0}{\hat r - r_0}, \qquad \kappa = \frac{L}{\mu}.
\end{equation}
\end{theorem}
\begin{proof}
From $r_0 = \|x^{(0)} - x^*\|_2 < \min\{\delta, \hat r\}$, we apply Theorem~\ref{thm:exact-eig} together with Lemma~\ref{lem:series}(b) by setting $q = \frac{2\mu}{L+\mu} \in (0,1)$ and $c = \frac{\mu M}{(L+\mu)\epsilon} > 0$ to obtain the result.
\end{proof}

\subsection{Cases with inexact eigenvectors}
In this subsection, we address the case of inexact eigenvectors, which covers most scenarios in practical computations.
We assume that the exact orthonormal eigenvectors of $\nabla^2 E(x^{(n)})$ are $\{u_i^{(n)}\}_{i=1}^d$, corresponding to the eigenvalues $\{\lambda_i^{(n)}\}_{i=1}^d$ sorted in ascending order, while the orthonormal directions $\{v_i^{(n)}\}_{i=1}^d$
are applied as ascent directions in the iteration scheme of SCS-HiSD:
\begin{equation}\label{eq:inexact_iter}
\begin{aligned}
x^{(n+1) } 
&= x^{(n)} - \beta_n \bigg( I - \sum_{i=1}^{k} v_i^{(n)} v_i^{(n)\top} - \sum_{i=1}^{k}\frac{\eta_n}{\left|\alpha_i^{(n)}\right|} v_i^{(n)} v_i^{(n)\top} \bigg) \nabla E(x^{(n)}) \\
&
= x^{(n)} - \beta_n \left(I - V_k^{(n)} (I+Y_k^{(n)}) V_k^{(n)\top} \right) \nabla E\big(x^{(n)}\big),
\end{aligned}
\end{equation}
where $Y_k^{(n)} = \mathrm{diag} \left( \zeta_1^{(n)}, \dots, \zeta_k^{(n)} \right) \in \mathbb{R}^{k\times k}$, $\zeta_i^{(n)}=\eta_n / |\alpha_i^{(n)}|$. 
Here, we recall that $\alpha_i^{(n)} = v_i^{(n)\top} \nabla^2 E(x^{(n)})v_i^{(n)}$.
We introduce an assumption to control the discrepancy between the exact eigenvectors $u_i^{(n)}$ and the orthonormal directions $v_i^{(n)}$ in the SCS-HiSD iteration scheme \eqref{eq:inexact_iter}.
\begin{assumption}\label{ass:2}
Assume that an error bound $\alpha\in(0,1)$ holds for all $n\geqslant 0$:
\begin{equation}\label{eq:diff_v}
\left\lVert u_i^{(n)} - v_i^{(n)} \right\rVert_2 \leqslant 
\frac{\alpha}{\sqrt{k}}, \quad i=1,\dots,k.
\end{equation}
\end{assumption}
Under \Cref{ass:2}, we can estimate the deviation between the subspace spanned by the exact eigenvectors $U_k^{(n)} = \left[u_1^{(n)}, \dots, u_k^{(n)}\right]$ and that spanned by their orthonormal numerical approximations $V_k^{(n)} = \left[v_1^{(n)}, \dots, v_k^{(n)}\right]$ by introducing the following lemma with $U_{-k}^{(n)} = \left[u_{k+1}^{(n)}, \dots, u_d^{(n)}\right]$.
\begin{lemma}\label{lem:represent}
Under \Cref{ass:2}, for the column-orthogonal matrix $V_k^{(n)}$ with the following representation:
\begin{equation}\label{eq:decomp}
V_k^{(n)}
= U^{(n)}C^{(n)}
= \left[ U_k^{(n)}, U_{-k}^{(n)} \right]
\begin{bmatrix}
C_k^{(n)}\\[2pt]
C_{-k}^{(n)}
\end{bmatrix}
=U_k^{(n)}C_k^{(n)}+U_{-k}^{(n)}C_{-k}^{(n)},
\end{equation}
where $C_k^{(n)}= U_k^{(n)\top} V_k^{(n)}\in\mathbb{R}^{k\times k}$ and $C_{-k}^{(n)}= U_{-k}^{(n)\top} V_k^{(n)}\in\mathbb{R}^{(d-k)\times k}$, we have
\begin{equation}
\lVert C_{-k}^{(n)} \rVert_2 \leqslant \alpha,
\qquad 
(1-\alpha) I  \preceq C_k^{(n)} C_k^{(n)\top}  \preceq I.
\end{equation}
\end{lemma}

\begin{proof}
We adopt the following distance based on orthogonal projection:
\begin{equation}
d\left(\mathrm{span}\{U_k^{(n)}\}, \mathrm{span}\{V_k^{(n)}\}\right)
:= \left\lVert U_k^{(n)} U_k^{(n)\top} - V_k^{(n)} V_k^{(n)\top} \right\rVert_2.
\end{equation}
To estimate this distance, we note that the term $\bigl\lVert Z_1^{\top} W_2 \bigr\rVert_2$ in Lemma~\ref{lem:proj-gap} is equivalent to the norm of the projection of $Z_1$ onto the orthogonal complement of $W_1$. 
Specifically, we have
\begin{equation*}
\left\lVert W_2^\top Z_1 \right\rVert_2^2 
= \left\lVert Z_1^\top (W_2 W_2^\top) Z_1 \right\rVert_2 
= \left\lVert Z_1^\top (I - W_1 W_1^\top) Z_1 \right\rVert_2 
= \left\lVert (I - W_1 W_1^\top) Z_1 \right\rVert_2^2.
\end{equation*}
Substituting column-orthogonal matrices $U_k^{(n)}$ and $V_k^{(n)}$ yields:
\begin{equation}\label{eq:subspace_equiv}
\left\lVert U_k^{(n)} U_k^{(n)\top} - V_k^{(n)} V_k^{(n)\top} \right\rVert_2 = 
\left\lVert \left(I - U_k^{(n)} U_k^{(n)\top}\right) V_k^{(n)} \right\rVert_2.
\end{equation}
Utilizing the property $\left(I - U_k^{(n)} U_k^{(n)\top}\right) U_k^{(n)} = O$, we can estimate the right-hand side of \eqref{eq:subspace_equiv} by introducing a difference term:
\begin{equation}
\begin{aligned}
&\left\lVert (I - U_k^{(n)} U_k^{(n)\top}) V_k^{(n)} \right\rVert_2
= \left\lVert (I - U_k^{(n)} U_k^{(n)\top}) (V_k^{(n)} - U_k^{(n)}) \right\rVert_2 \\
\leqslant \;
&\left\lVert I - U_k^{(n)} U_k^{(n)\top} \right\rVert_2
\left\lVert V_k^{(n)} - U_k^{(n)} \right\rVert_2 
\leqslant \left\lVert V_k^{(n)} - U_k^{(n)} \right\rVert_2
\leqslant \left\lVert V_k^{(n)} - U_k^{(n)} \right\rVert_F.
\end{aligned}
\end{equation}
Finally, based on the vector error bound \eqref{eq:diff_v} in \Cref{ass:2}, we obtain:
\begin{equation}
\left\lVert V_k^{(n)} - U_k^{(n)} \right\rVert_F^2 = 
\sum_{i=1}^k \left\lVert v_i^{(n)} - u_i^{(n)} \right\rVert_2^2
\leqslant 
k \cdot \frac{\alpha^2}{k} = \alpha^2,
\end{equation}
which indicates that the subspace distance satisfies
\begin{equation}\label{eq:dis_subspace}
\left\lVert U_k^{(n)} U_k^{(n)\top} - V_k^{(n)} V_k^{(n)\top} \right\rVert_2 \leqslant \alpha.
\end{equation}
Combining \eqref{eq:dis_subspace} with Lemma \ref{lem:proj-gap}, we obtain
\begin{equation}\label{eq:vector_alpha}
\left\lVert C_{-k}^{(n)} \right\rVert_2 = 
\left\lVert U_{-k}^{(n)\top} V_k^{(n)} \right\rVert_2 = \left\lVert U_k^{(n)} U_k^{(n)\top} - V_k^{(n)} V_k^{(n)\top} \right\rVert_2
\leqslant \alpha.
\end{equation}

Based on \eqref{eq:dis_subspace}, we have
$U_k^{(n)} U_k^{(n)\top} - \alpha I \preceq V_k^{(n)} V_k^{(n)\top}$,
and consequently,
\begin{equation*}
(1-\alpha) I = U_k^{(n)\top}\left(U_k^{(n)} U_k^{(n)\top}-\alpha I\right) U_k^{(n)}
\preceq
C_k^{(n)} C_k^{(n)\top}.
\end{equation*} 
Using the column-orthogonality property $\|V_k^{(n)}\|_2=\|U_k^{(n)}\|_2=1$, we obtain
\begin{equation}
\left\|C_k^{(n)} C_k^{(n)\top}\right\|_2
\leqslant \left\|C_k^{(n)}\right\|_2^2
\leqslant \left\|U_k^{(n)}\right\|_2^2  \left\|V_k^{(n)}\right\|_2^2
= 1,
\end{equation}
so we conclude that $C_k^{(n)} C_k^{(n)\top} \preceq I$, which completes the proof.
\end{proof}

According to Lemma \ref{lem:represent}, we can linearly express $v_i^{(n)}$, the $i$-th column of $V_k^{(n)}$, using basis $\{u_j^{(n)}\}_{j=1}^d$ as $v_i^{(n)} = \sum_{j=1}^{d}c_{ji}^{(n)}u_j^{(n)}$.
Based on \eqref{eq:diff_v}, we have
\begin{equation}\label{eq:neq}
\biggl\lVert (1-c_{ii}^{(n)})u_i^{(n)} + \sum_{j\neq i}c_{ji}^{(n)}u_j^{(n)} \biggr\rVert_2 \leqslant \frac{\alpha}{\sqrt{k}}.
\end{equation}
Given the orthonormality of 
$\{u_j^{(n)}\}_{j=1}^d$ and combined with the inequality \eqref{eq:neq}, we deduce
$\sum\limits_{j\neq i}\left(c_{ji}^{(n)}\right)^2 \leqslant \frac{\alpha^2}{k}$ and 
$\left(c_{ii}^{(n)}\right)^2 \geqslant 1-\frac{\alpha^2}{k}$.
For the one-step iteration \eqref{eq:inexact_iter} with inexact eigenvectors, we have the following theorem. 

\begin{theorem}\label{thm:approx-eig}
Under Assumptions~\ref{ass:1} and \ref{ass:2} where $\alpha$ satisfies
\begin{equation}\label{eq:alpha1}
\alpha\leqslant\sqrt{\frac{k\epsilon}{2(L+\epsilon)}},\quad
\alpha + \kappa(\alpha+ 2p(\alpha)) < 2,  
\end{equation}
if for some $n \geqslant 0$, $r_n = \|x^{(n)} - x^*\|_2 < \delta$ holds for \eqref{eq:inexact_iter} with  $\beta_n=\dfrac{4}{b(\alpha)}$ and $\eta_n=\dfrac{3}{2}\mu$,
then the following estimate holds:
\begin{equation}\label{eq:inexact_a}
r_{n+1} = \|x^{(n+1)} - x^*\|_2 \leqslant (1 - q(\alpha) ) r_n + c(\alpha) r_n^2.
\end{equation}
Here $p(\alpha)>0$, $b(\alpha)>0$, $q(\alpha)\in(0,1)$, $c(\alpha)>0$ are defined as follows,
\begin{equation}\label{eq:pbqc}
\begin{aligned}
p(\alpha)&=\left(\frac{5}{2}+\frac{27\mu}{2\epsilon}\right) \alpha+\left(\frac{3\mu}{\epsilon}+\frac{1}{2}\right) \alpha^2,
\quad 
b(\alpha) = (2-\alpha-\alpha^2)L+3(2-\alpha)\mu, \\
q(\alpha) &= \frac{(4 - 2\alpha)\mu - (2\alpha+4p(\alpha))L}{b(\alpha)}, 
\quad c(\alpha) = \frac{6M\mu}{b(\alpha)\epsilon}, 
\quad \kappa = \frac{L}{\mu}.
\end{aligned}
\end{equation}
\end{theorem}

\begin{proof}
From \eqref{eq:inexact_iter}, we denote $A^{(n)} = I - V_k^{(n)} (I+Y_k^{(n)}) V_k^{(n)\top}$, and then $\lVert A^{(n)} \rVert_2\leqslant \max\{1,\lVert Y_k^{(n)} \rVert_2\}$.
First, we establish an upper bound for $\lVert Y_k^{(n)} \rVert_2 = \max\{|\zeta_1^{(n)}|,$ \dots, $|\zeta_k^{(n)}|\}$.
For $\alpha\leqslant \sqrt{\frac{k\epsilon}{2(L+\epsilon)}}$, we have
\begin{equation}\label{eq:bound_alpha}
\begin{aligned}
\alpha_i^{(n)} &= \sum_{j=1}^d \left(c_{ji}^{(n)}\right)^2\lambda_j^{(n)} \leqslant
\left(c_{ii}^{(n)}\right)^2\lambda_i^{(n)}+\sum_{j=k+1}^d \left(c_{ji}^{(n)}\right)^2\lambda_j^{(n)} \\
&\leqslant
-\Bigl(1-\frac{\alpha^2}{k}\Bigr)\epsilon+\frac{\alpha^2}{k}L = -\epsilon+\frac{\alpha^2}{k}
(L+\epsilon)\leqslant
-\frac{\epsilon}{2},\quad i=1,\dots,k.
\end{aligned}
\end{equation}
Therefore, $\lVert Y_k^{(n)} \rVert_2 \leqslant \frac{3\mu}{\epsilon}$ and $\lVert A^{(n)} \rVert_2 \leqslant \frac{3\mu}{\epsilon}$.
Following a similar procedure as in Theorem~\ref{thm:exact-eig}, we define $Q^{(n)}$ and $B^{(n)}$ as in \eqref{eq:qnbn} and obtain
\begin{equation}
x^{(n+1) }-x^* = \left(Q^{(n)}+B^{(n)}\right) (x^{(n)}-x^*).
\end{equation}
From 
$\lVert B^{(n)} \rVert_2 
\leqslant \tfrac12 \beta_n M \lVert A^{(n)} \rVert_2  r_n
\leqslant \tfrac{3M\mu}{2\epsilon} \beta_n r_n$, we obtain the following inequality
\begin{equation}\label{eq:inexact_b}
r_{n+1}\leqslant
\lVert Q^{(n)} \rVert_2 r_n + \lVert B^{(n)} \rVert_2 r_n
\leqslant
\lVert Q^{(n)} \rVert_2 r_n +  \tfrac{3M\mu}{2\epsilon} \beta_n r_n^2.
\end{equation}

To estimate the upper bound of $\lVert Q^{(n)} \rVert_2$, we utilize the spectral decomposition of the Hessian matrix in the following form:
\begin{equation}
\nabla^2 E(x^{(n)}) = U_k^{(n)}\Lambda_k^{(n)}U_k^{(n)\top} + U_{-k}^{(n)}\Lambda_{-k}^{(n)}U_{-k}^{(n)\top},
\end{equation}
where $\Lambda_k^{(n)}=\mathrm{diag}(\lambda_1^{(n)},\dots,\lambda_k^{(n)})$ and $\Lambda_{-k}^{(n)}=\mathrm{diag}(\lambda_{k+1}^{(n)},\dots,\lambda_d^{(n)})$.
With \eqref{eq:decomp}, we deduce that:
\begin{equation*}
\begin{aligned}
A^{(n)} &= I - U_k^{(n)} C_k^{(n)} \bigl(I+Y_k^{(n)}\bigr) C_k^{(n)\top} U_k^{(n)\top} 
- U_{-k}^{(n)} C_{-k}^{(n)} \bigl(I+Y_k^{(n)}\bigr) C_{-k}^{(n)\top} U_{-k}^{(n)\top} \\
&\quad - U_{-k}^{(n)} C_{-k}^{(n)} \bigl(I+Y_k^{(n)}\bigr) C_k^{(n)\top} U_k^{(n)\top}
- U_k^{(n)} C_k^{(n)} \bigl(I+Y_k^{(n)}\bigr) C_{-k}^{(n)\top} U_{-k}^{(n)\top},
\end{aligned}
\end{equation*}
and rewrite $A^{(n)} \nabla^2 E(x^{(n)}) = K_0^{(n)} + R_0^{(n)}$, where
\begin{equation*}
\begin{aligned}
K_0^{(n)} &=  
U_k^{(n)} \bigl(I - C_k^{(n)} C_k^{(n)\top} \bigr) \Lambda_k^{(n)} U_k^{(n)\top} 
+ U_{-k}^{(n)} \bigl(I - C_{-k}^{(n)} C_{-k}^{(n)\top}\bigr)\Lambda_{-k}^{(n)}U_{-k}^{(n)\top}\\
&\quad - U_k^{(n)}  C_k^{(n)} Y_k^{(n)} C_k^{(n)\top}\Lambda_k^{(n)} U_k^{(n)\top},\\
R_0^{(n)} &= - U_{-k}^{(n)}C_{-k}^{(n)} Y_k^{(n)} C_{-k}^{(n)\top}\Lambda_{-k}^{(n)} U_{-k}^{(n)\top}
- U_k^{(n)} C_k^{(n)} \bigl(I + Y_k^{(n)}\bigr) C_{-k}^{(n)\top}\Lambda_{-k}^{(n)} U_{-k}^{(n)\top}\\
&\quad - U_{-k}^{(n)} C_{-k}^{(n)} \bigl(I + Y_k^{(n)}\bigr) C_k^{(n)\top}\Lambda_k^{(n)} U_k^{(n)\top}.
\end{aligned}
\end{equation*}
We adopt the following splitting strategy 
\begin{align*}
I - C_k^{(n)} C_k^{(n)\top}
&= \frac{\alpha}{2} I +\left(
\Bigl(1-\frac{\alpha}{2}\Bigr) I - 
C_k^{(n)} C_k^{(n)\top}
\right),\\
I - C_{-k}^{(n)} C_{-k}^{(n)\top}
&= \Bigl(1-\frac{\alpha^{2}}{2}\Bigr)I + \left(
\frac{\alpha^{2}}{2} I - 
C_{-k}^{(n)} C_{-k}^{(n)\top} 
\right),\\
 C_k^{(n)} Y_k^{(n)} 
&=  Y_k^{(n)} C_k^{(n)} +  \Bigl(C_k^{(n)} Y_k^{(n)}-Y_k^{(n)}C_k^{(n)}\Bigr) ,
\end{align*}
for further decomposing $K_0^{(n)} = K^{(n)}+R^{(n)}$:
\begin{equation*}
\begin{aligned}
K^{(n)} &=
U_k^{(n)}\left(\frac{\alpha}{2}I - \Bigl(1-\frac{\alpha}{2}\Bigr) Y_k^{(n)} \right) \Lambda_k^{(n)} U_k^{(n)\top} + \Bigl(1-\frac{\alpha^{2}}{2}\Bigr) U_{-k}^{(n)} \Lambda_{-k}^{(n)} U_{-k}^{(n)\top},\\
R^{(n)} &=U_k^{(n)}\bigl(I+Y_k^{(n)}\bigr)\left(\Bigl(1-\frac{\alpha}{2}\Bigr) I - C_k^{(n)} C_k^{(n)\top}\right) \Lambda_k^{(n)} U_k^{(n)\top}\\
&\quad+U_{-k}^{(n)}\left(\frac{\alpha^{2}}{2} I - C_{-k}^{(n)} C_{-k}^{(n)\top}\right) \Lambda_{-k}^{(n)} U_{-k}^{(n)\top}\\
&\quad-U_k^{(n)} \Bigl(C_k^{(n)} Y_k^{(n)}-Y_k^{(n)}C_k^{(n)}\Bigr) C_k^{(n)\top} \Lambda_k^{(n)} U_k^{(n)\top}.
\end{aligned}
\end{equation*}
This leads to an upper bound of $\lVert Q^{(n)} \rVert_2= \lVert I- \beta_n A^{(n)} \nabla^2 E(x^{(n)}) \rVert_2$:
\begin{equation}\label{eq:qnleq}
\lVert Q^{(n)} \rVert_2 \leqslant
\bigl\lVert I - \beta_n K^{(n)} \bigr\rVert_2 + \beta_n \left(\lVert R_0^{(n)} \rVert_2 +\lVert R^{(n)} \rVert_2\right).
\end{equation}

For $K^{(n)}$, it is straightforward that its eigenvalues are 
\begin{equation}
\begin{aligned}
z_i^{(n)} &:= \frac{\alpha}{2} \lambda_i^{(n)} -\left(1-\frac{\alpha}{2}\right) \frac{3\mu}{2|\alpha_i^{(n)}|}\lambda_i^{(n)}, &1\leqslant i\leqslant k,\\
z_i^{(n)} &:= \left(1-\frac{\alpha^2}{2}\right) \lambda_j^{(n)}, &k< i\leqslant d. 
\end{aligned}
\end{equation}
Similar to \eqref{eq:bound_alpha}, we can bound $\alpha_i^{(n)}$ for $1\leqslant i\leqslant k$ by
\begin{align*}
\alpha_i^{(n)} &= \left(c_{ii}^{(n)}\right)^2\lambda_i^{(n)} + \sum_{j\neq i} \left(c_{ji}^{(n)}\right)^2\lambda_j^{(n)} 
\geqslant \lambda_i^{(n)} - L\sum_{j\neq i} \left(c_{ji}^{(n)}\right)^2 
\geqslant \lambda_i^{(n)} - \frac{\alpha^2}{k}L \\
&\geqslant \lambda_i^{(n)} - \frac{L\epsilon}{2(L+\epsilon)} 
> \lambda_i^{(n)} - \frac{\epsilon}{2} 
\geqslant \lambda_i^{(n)} + \frac{1}{2}\lambda_i^{(n)}
= \frac{3}{2}\lambda_i^{(n)},\\
\alpha_i^{(n)} &= \left(c_{ii}^{(n)}\right)^2\lambda_i^{(n)} + \sum_{j\neq i} \left(c_{ji}^{(n)}\right)^2\lambda_j^{(n)} 
\leqslant  \Bigl(1-\frac{\alpha^2}{k}\Bigr)\lambda_i^{(n)} + \frac{\alpha^2}{k}L \\
&\leqslant \lambda_i^{(n)} + \frac{\epsilon}{2(L+\epsilon)}(L-\lambda_i^{(n)})
\leqslant\lambda_i^{(n)} - \frac{1}{2}\lambda_i^{(n)} = \frac{1}{2}\lambda_i^{(n)}.
\end{align*}
This indicates $\left(1-\frac{\alpha}{2}\right) \mu-\frac{\alpha}{2}L \leqslant z_i < 3\left(1-\frac{\alpha}{2}\right)\mu$. 
Then we have a uniform bound for all the eigenvalues of $K^{(n)}$:
\begin{equation}
\Bigl(1-\frac{\alpha}{2}\Bigr) \mu-\frac{\alpha}{2}L 
\leqslant z_i \leqslant
\Bigl(1-\frac{\alpha^2}{2}\Bigr) L + 2\left(1-\frac{\alpha}{2}\right)\mu,
\quad 1\leqslant i\leqslant d.
\end{equation}
Following the same procedure in Theorem~\ref{thm:exact-eig}, we obtain
\begin{equation}\label{eq:imbnk}
\bigl\lVert I-\beta_n K^{(n)} \bigr\rVert_2
\leqslant \frac{(2+\alpha-\alpha^2) L+(2-\alpha) \mu}{b(\alpha)},\quad
\beta_n=\frac{4}{b(\alpha)}.
\end{equation}

For $D_k^{(n)} = C_k^{(n)} - I$, by utilizing the orthonormality of $\{u_j\}$ and the bound established in \eqref{eq:neq}, we have
\begin{equation}
\lVert D_k^{(n)} \rVert_2^2 \leqslant \lVert D_k^{(n)} \rVert_F^2 = \sum_{i=1}^k \bigg( 
\left(c_{ii}^{(n)}-1\right)^2 + \sum_{j=1, j\neq i}^d \left(c_{ji}^{(n)}\right)^2 
\bigg) \leqslant \sum_{i=1}^k \frac{\alpha^2}{k} = \alpha^2,
\end{equation}
so $\lVert D_k^{(n)} \rVert_2 \leqslant \alpha$ and 
$\lVert C_k^{(n)}Y_k^{(n)} - Y_k^{(n)}C_k^{(n)} \rVert_2 
= 
\lVert D_k^{(n)} Y_k^{(n)}-Y_k^{(n)}D_k^{(n)} \rVert_2 
\leqslant 
\frac{6\mu}{\epsilon}\alpha$. 
Finally, we come to estimate $\lVert R_0^{(n)} \rVert_2$  
and $\lVert R^{(n)} \rVert_2$: 
\begin{align*}
\lVert R_0^{(n)} \rVert_2 &\leqslant 
\left\lVert U_{-k}^{(n)}C_{-k}^{(n)} Y_k^{(n)} C_{-k}^{(n)\top}\Lambda_{-k}^{(n)} U_{-k}^{(n)\top} \right\rVert_2 \\
&\quad +\left\lVert U_k^{(n)} C_k^{(n)} \bigl(I + Y_k^{(n)}\bigr) C_{-k}^{(n)\top}\Lambda_{-k}^{(n)} U_{-k}^{(n)\top} \right\rVert_2 \\
&\quad +\left\lVert U_{-k}^{(n)} C_{-k}^{(n)} \bigl(I + Y_k^{(n)}\bigr) C_k^{(n)\top}\Lambda_k^{(n)} U_k^{(n)\top} \right\rVert_2 \\
&\leqslant L \left\lVert C_{-k}^{(n)} Y_k^{(n)} C_{-k}^{(n)\top} \right\rVert_2 
+2L \left\lVert C_k^{(n)}\bigl(I+Y_k^{(n)}\bigr) C_{-k}^{(n)\top} \right\rVert_2\\
&\leqslant \frac{3\mu}{\epsilon}L\alpha^2 +2\Bigl(1+\frac{3\mu}{\epsilon}\Bigr) L\alpha,\\
\lVert R^{(n)} \rVert_2 &\leqslant 
\left \lVert U_k^{(n)}\bigl(I+Y_k^{(n)}\bigr)\left(\Bigl(1-\frac{\alpha}{2}\Bigr) I - C_k^{(n)} C_k^{(n)\top}\right) \Lambda_k^{(n)} U_k^{(n)\top} \right \rVert_2\\
&\quad +\left \lVert U_{-k}^{(n)}\left(\frac{\alpha^{2}}{2} I - C_{-k}^{(n)} C_{-k}^{(n)\top}\right) \Lambda_{-k}^{(n)} U_{-k}^{(n)\top} \right \rVert_2\\
&\quad +\left \lVert U_k^{(n)} \Bigl(C_k^{(n)} Y_k^{(n)}-Y_k^{(n)}C_k^{(n)}\Bigr) C_k^{(n)\top} \Lambda_k^{(n)} U_k^{(n)\top} \right \rVert_2\\
&\leqslant  
L \left\lVert \bigl(I+Y_k^{(n)}\bigr)\Bigl(\bigl(1-\frac{\alpha}{2}\bigr) I - C_k^{(n)} C_k^{(n)\top}\Bigr) \right\rVert_2 \\
&\quad + L \left\lVert \frac{\alpha^{2}}{2} I - C_{-k}^{(n)} C_{-k}^{(n)\top} \right\rVert_2
+ L \left\lVert \bigl(C_k^{(n)} Y_k^{(n)}-Y_k^{(n)}C_k^{(n)}\bigr) C_k^{(n)\top} \right\rVert_2\\
& \leqslant \frac{\alpha}{2}\Bigl(1+\frac{3\mu}{\epsilon}\Bigr)L
+ \frac{\alpha^2}{2}L
+ \frac{6\mu}{\epsilon}L\alpha.
\end{align*}
Therefore, we have
\begin{equation}\label{eq:rr0}
\lVert R^{(n)}\rVert_2 + \lVert R_0^{(n)}\rVert_2 
\leqslant \left( \left(\frac{5}{2}+\frac{27\mu}{2\epsilon}\right) \alpha + \left(\frac{3\mu}{\epsilon}+\frac{1}{2}\right) \alpha^2 \right)L 
= p(\alpha)L.
\end{equation}
From \eqref{eq:imbnk} and \eqref{eq:rr0}, we can estimate $\lVert Q^{(n)} \rVert_2$ from \eqref{eq:qnleq}:
\begin{equation}\label{eq:inexact_c}
\bigl\lVert Q^{(n)} \bigr\rVert_2 \leqslant \frac{(2+\alpha-\alpha^2) L+(2-\alpha) \mu+4p(\alpha)L}{b(\alpha)}.
\end{equation}
Finally, substituting \eqref{eq:inexact_c} into \eqref{eq:inexact_b} yields \eqref{eq:inexact_a}.
\end{proof}

Based on \Cref{thm:approx-eig} and Lemma~\ref{lem:series}, we obtain the following conclusion regarding the linear convergence rate of SCS-HiSD in the case of inexact eigenvectors.
For sufficiently small $\alpha>0$, the linear convergence rate $\frac{1}{1+q(\alpha)}$ is independent of $\epsilon$. 
\begin{theorem}\label{thm:inexact-conclusion}
Under Assumptions~\ref{ass:1} and \ref{ass:2}, if $\alpha$ satisfies \eqref{eq:alpha1}, the initial guess $x^{(0)}$ satisfies
\begin{equation}
r_0 = \lVert x^{(0)} - x^* \rVert_2 < \min\{\delta, \hat{r}\}, \quad \hat{r} = \frac{q(\alpha)}{c(\alpha)},
\end{equation}
and the parameters are set as $\beta_n=\frac{4}{b(\alpha)}$ and $\eta_n=\frac{3}{2}\mu$ for all $n\geqslant0$, 
then the sequence $x^{(n)}$ defined by \eqref{eq:inexact_iter} converges to $x^*$ as $n \to \infty$ with a linear convergence rate:
\begin{equation}
r_n = \lVert x^{(n)} - x^* \rVert_2 \leqslant
\left(\frac{1}{1+q(\alpha)}\right)^n \frac{\hat{r} r_0}{\hat{r} - r_0}.
\end{equation}
Here, $b(\alpha)$, $q(\alpha)$ and $c(\alpha)$ are defined as in \eqref{eq:pbqc}.
\end{theorem}

The parameter $\eta$ plays an important role in the performance of the proposed method. 
A small $\eta$ yields limited acceleration, whereas a large $\eta$ may compromise numerical stability and lead to divergence.
From the theoretical results in \Cref{thm:exact_conclusion,thm:inexact-conclusion}, $\eta$ should be of the same order as the smallest positive eigenvalue of the Hessian to achieve a better convergence  behavior.
In \Cref{alg:hisd,alg:qnhisd}, $v_1^{(n)}, \dots, v_k^{(n)}$ approximate the orthonormal eigenvectors corresponding to the smallest $k$ eigenvalues, which are all negative near the index-$k$ saddle point. 
Nevertheless, many eigenvector solvers (denoted by ``\texttt{EigenSol}'' in \Cref{alg:hisd,alg:qnhisd}), e.g., LOBPCG, typically provide approximate information about several of the smallest positive eigenvalues during computations. 
Therefore, the approximation of the smallest positive eigenvalue is readily available within the algorithm implementation without much additional computational cost.

\section{Numerical experiments}\label{sec:5}
In this section, we conduct multiple numerical experiments to compare the performance of SCS-HiSD with the original HiSD, thereby demonstrating the efficiency of the proposed method. 
For numerical stability, the denominator $|\alpha_i^{(n)}|$ in \Cref{alg:qnhisd} is regularized as $|\alpha_i^{(n)}|_+ = \max\{|\alpha_i^{(n)}|, \varepsilon\}$ with a small positive parameter $\varepsilon=10^{-4}$. 
In all experiments, the ascent directions $v_i^{(n)}$ are updated using the one-step LOBPCG method in each iteration.

\subsection{Modified Strictly Convex 2 function}
As the first numerical example, we consider the modified Strictly Convex 2 function \cite{raydan1997barzilai}, defined as
\begin{equation} 
E(x) = \frac{1}{10} \sum_{i=1}^{d} s_i a_i \left(\exp(x_i) - x_i\right), \qquad a_i = 5i-4,
\end{equation} 
where $s_i = -1$ for $1 \leqslant i \leqslant 5$ and $s_i = 1$ for $6 \leqslant i \leqslant d$.
We set the dimension to $d = 100$ and aim to locate an index-$5$ saddle point $x^* = [0, \dots, 0]^{\top}$ from $x^{(0)} = [-6, \dots, -6]^{\top}$, using a fixed step size $\beta_0 = 0.02$, which satisfies the step-size condition in \Cref{thm:exact-eig}. 
We set $\eta = 2.6$, an empirical value that approximates the smallest positive eigenvalue of the Hessian at ${x}^*$. 
A smaller value of $\eta$ ensures convergence but provides weaker acceleration. 
As illustrated in \Cref{fig:err_strictly}, SCS-HiSD converges to $x^*$ substantially faster than HiSD, validating the effectiveness of the proposed method.

\begin{figure}[htbp]
\centering
\includegraphics[width=0.45\linewidth]{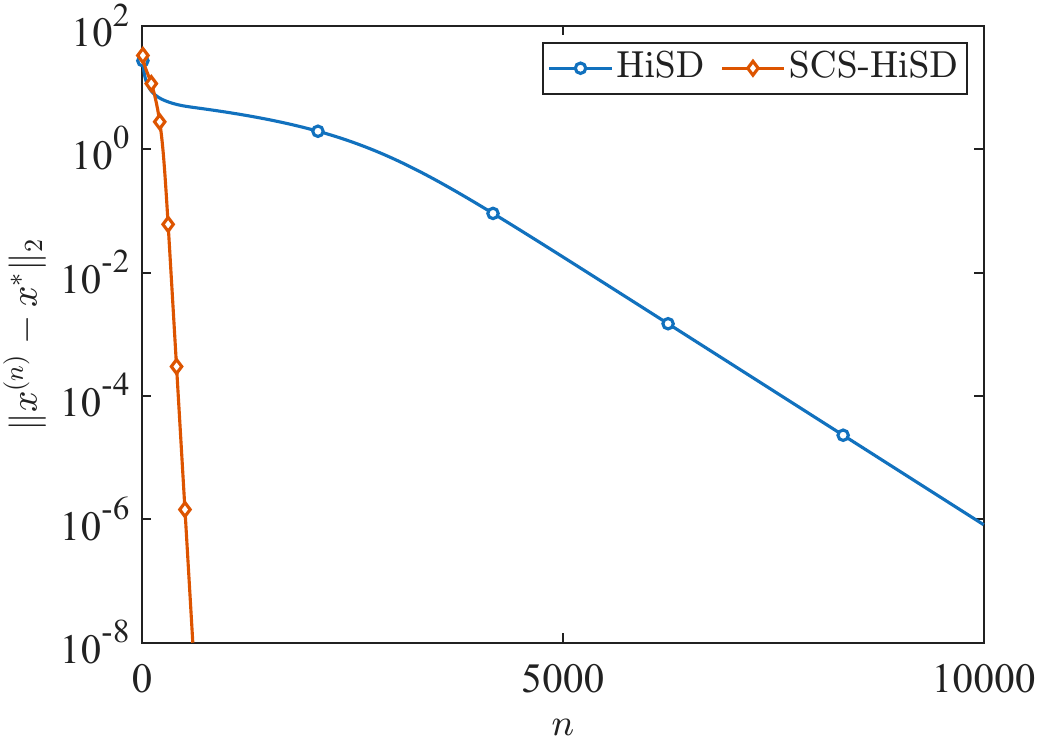}
\caption{ \label{fig:err_strictly} 
Evolution of the error $\|x^{(n)}-x^*\|_2$ with respect to the iteration number $n$ for the modified Strictly Convex 2 function.}
\end{figure}

\subsection{Modified Rosenbrock-type function}
In this example, we consider the Rosenbrock-type function \cite{rosenbrock1960automatic} with modified terms \cite{luo2022convergence},
\begin{equation}
E(x) = \sum_{i=1}^{d-1} \Bigl( 100 \bigl(x_{i+1} - x_i^{2}\bigr) ^{2} + (1 - x_i) ^{2} \Bigr) + \sum_{i=1}^{d} s_i  \arctan^2(x_i - 1).
\end{equation}
We set the dimension to $d = 1000$ and specify the parameters $s_i = 1$ for $6\leqslant i <d$ and $s_d = 150$. 
For $1 \leqslant i \leqslant 5$, we consider three cases: (a) $s_i = -662$, (b) $s_i = -659$, and (c) $s_i = -657.5$.
In all cases, $x^\ast = [1, \dots, 1]^{\top}$ is an index-4 saddle point, and the initialization is $x^{(0)} = x^\ast + r/{\|r\|_2}$ with a fixed $r \sim \mathcal{N}(0, I_d)$.
The spectral properties of the Hessian at $x^*$ differ significantly among the three cases.
The largest eigenvalue in each case is about 1800. 
The negative eigenvalue closest to zero is $-8.89$ in case (a), $-3.24$ in case (b), and $-0.41$ in case (c).
Therefore, the ill-conditioning becomes increasingly severe from case (a) to (c), leading to progressively slower dynamics in the unstable subspace.
This example provides a useful testbed for examining how the performance of different algorithms changes as the problem becomes increasingly ill-conditioned.

In addition to the original HiSD, we compare SCS-HiSD with the heavy-ball accelerated HiSD method (A-HiSD) \cite{luo2025accelerated}.
The corresponding update is given by
\begin{equation}
    x^{(n+1)} = x^{(n)} + \beta^{(n)} g^{(n)} + \gamma_{\mathrm{M}} \bigl(x^{(n)} - x^{(n-1)}\bigr),
\end{equation}
with $x^{(-1)}=x^{(0)}$, where $g^{(n)}$ denotes the HiSD search direction and $\gamma_{\mathrm{M}}$ is the momentum parameter. 
Besides fixed step sizes, we also consider BB step sizes,
$\beta^{(n)} = \left| {\langle\Delta g^{(n)}, \Delta x^{(n)}\rangle}/{\langle\Delta g^{(n)}, \Delta g^{(n)}\rangle} \right| $, 
where $\Delta g^{(n)}= g^{(n)}-g^{(n-1)}$ and $\Delta x^{(n)} = x^{(n)}-x^{(n-1)}$ \cite{barzilai1988bb}.
In this experiment, the BB step sizes are clipped to $[\beta_0, 8\beta_0]$ for A-HiSD and $[0.5\beta_0, 8\beta_0]$ for the others.
In SCS-HiSD, we set $\eta = 180$, which approximates the smallest positive eigenvalue.

In numerical experiments, we tune the step sizes $\beta_0$ in each method and $\gamma_{\mathrm{M}}$ in A-HiSD for a fair comparison.
For all methods, we use $\beta_0 = 0.001$ for both fixed and BB step sizes.
For A-HiSD, we use $\gamma_{\mathrm{M}}=0.75$.
These parameters yield good performance, whereas larger choices may lead to divergence.

\begin{figure}[htbp]
\centering
\includegraphics[width=1\linewidth]{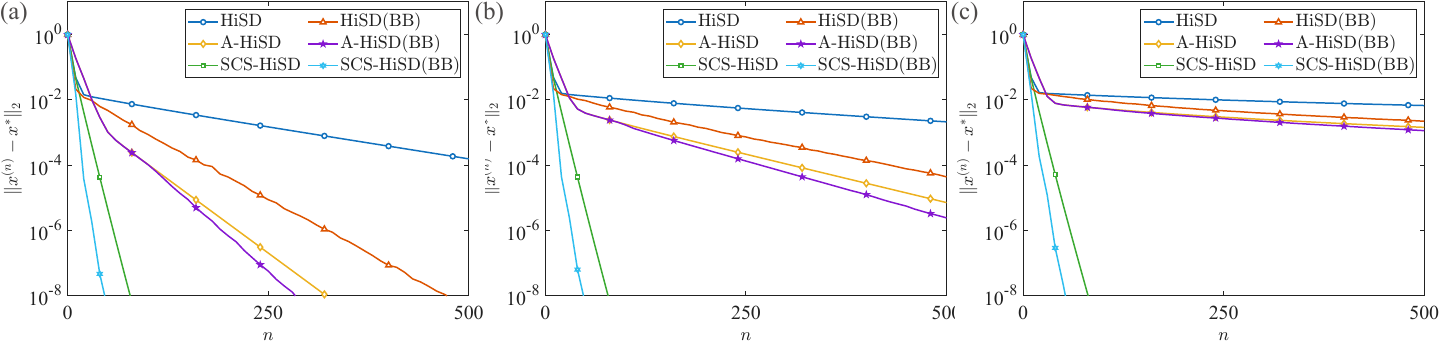}
\caption{\label{fig:err_Rosenbrock}
Evolution of the error $\|x^{(n)}-x^*\|_2$ with respect to the iteration number $n$ for the modified Rosenbrock-type function in cases (a)--(c).
The notation ``(BB)'' denotes the use of BB step sizes, while the others correspond to fixed step sizes.}
\end{figure}

As the small-magnitude negative eigenvalue approaches zero from case (a) to case (c), the baseline HiSD method exhibits increasingly severe stagnation, as shown in \Cref{fig:err_Rosenbrock}.
In case (a), which is moderately ill-conditioned, A-HiSD provides significant acceleration comparable to that of SCS-HiSD.
As the ill-conditioning becomes more severe in cases (b) and (c), the acceleration provided by A-HiSD becomes substantially less effective.
In particular, in case (c), A-HiSD exhibits stagnation behavior similar to that of the original HiSD method.
In contrast, the proposed SCS-HiSD method effectively overcomes this bottleneck and achieves significantly faster convergence in all three cases.
The convergence behavior of SCS-HiSD remains relatively stable as the problem becomes increasingly ill-conditioned.

These results demonstrate that SCS-HiSD achieves robust acceleration for ill-conditioned saddle points with small-magnitude negative eigenvalues. 
The original HiSD suffers from slow convergence in such cases, leading to high computational cost. 
For moderately ill-conditioned problems, A-HiSD and SCS-HiSD exhibit comparable performance.
However, as the ill-conditioning becomes severe, the acceleration provided by A-HiSD becomes limited and insufficient to overcome the resulting computational bottleneck.
In contrast, the proposed SCS-HiSD method remains consistently effective and achieves substantial acceleration even in severely ill-conditioned cases.

\subsection{Nematic liquid crystals confined in a square} 
In this numerical experiment, we consider rod-like nematic liquid crystals confined in a two-dimensional square domain \cite{kralj2011curvature,robinson2017molecular}. 
According to the Landau--de Gennes (LdG) theory  \cite{de1993physics}, the liquid-crystal system can be described by the $\mathsf{Q}$-tensor, a $2\times2$ symmetric traceless tensor field defined on the square domain $\Omega=[-1,1]^2$. 
We adopt the following nondimensionalized LdG free energy:
\begin{equation}\label{eq:ldg}
E[\mathsf{Q}]=\int_{\Omega}\left(\frac12|\nabla \mathsf{Q}|^2+\alpha 
\left(\frac{a}{4}|\mathsf{Q}|^2+\frac{1}{8}|\mathsf{Q}|^4\right)
\right)\mathrm d\mathbf r, 
\quad
\mathsf{Q}= 
\begin{bmatrix} Q_{11}(\mathbf{r}) & Q_{12}(\mathbf{r})\\
Q_{12}(\mathbf{r}) & -Q_{11}(\mathbf{r})\end{bmatrix},
\end{equation}
where $\alpha>0$ is a parameter related to the physical domain size \cite{yin2020construction}.
The parameter $a$ represents the reduced temperature and is taken as $a = -1.672$ to ensure that the system is in the nematic state. 
Strong anchoring (tangential Dirichlet) boundary conditions are imposed on $\partial\Omega$ to enforce alignment of the liquid crystals with the square boundary, 
\begin{equation}
\begin{aligned}
\mathsf{Q}(x=\pm1, y) = \dfrac{S_0}{2}
\begin{bmatrix} 1 & 0\\0 & -1\end{bmatrix}, \quad y \in (-1,1),
\\
\mathsf{Q}(x, y=\pm1) = \dfrac{S_0}{2}
\begin{bmatrix} -1 & 0\\0 & 1\end{bmatrix}, \quad x \in (-1,1),
\end{aligned}
\end{equation}
where $S_0 = \sqrt{2|a|}$. 
For large $\alpha$, these boundary conditions give rise to multiple local minima and saddle points of \eqref{eq:ldg}, making this system a prototypical example of a complex energy landscape.
The energy functional is discretized using 100 uniform grid points along each dimension. 
Previous studies have demonstrated the effectiveness of HiSD in locating saddle points and constructing solution landscapes.

\begin{figure}[htbp]
\centering
\includegraphics[width=0.9\linewidth]{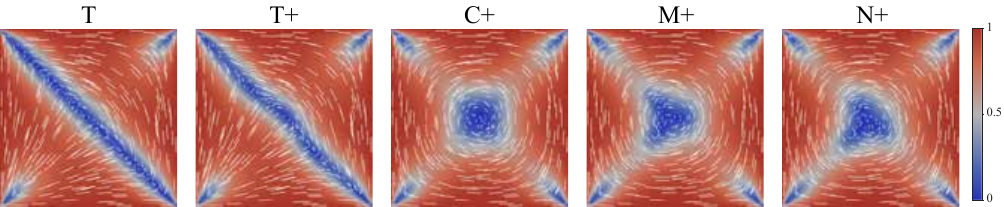}
\caption{\label{fig:nlc} 
Several saddle points of the liquid crystal example for $\alpha=50$. 
The color represents the relative magnitude of directional ordering $|S(\mathbf{r})|/S_0$, where $S(\mathbf{r})/2=\pm \sqrt{Q_{11}^2(\mathbf{r}) + Q_{12}^2(\mathbf{r})}$ represents the eigenvalues of $\mathsf{Q}(\mathbf{r})$.
The directions of white bars represent the nematic field directors $\mathbf{n}(\mathbf{r})=(\cos \theta(\mathbf{r}), \sin \theta(\mathbf{r}))$, where $\cos \theta(\mathbf{r})=\sqrt{1/2 + Q_{11}(\mathbf{r})/|S(\mathbf{r})|}$. 
}
\end{figure}

As the domain size parameter $\alpha$ increases, the system experiences a sequence of bifurcations (see Fig.~1 in the Supplemental Material of \cite{yin2020construction}), thereby generating a large number of stationary points. 
Mathematically, these bifurcations occur only when the Hessian at the corresponding stationary point possesses a zero eigenvalue. 
Consequently, near each bifurcation, the Hessian at the corresponding stationary point has eigenvalues close to zero, which substantially reduces the convergence rate of HiSD. 
For example, the $T$ state is an index-3 saddle point at $\alpha=45$ and an index-4 saddle point at $\alpha=50$. 
A supercritical pitchfork bifurcation occurs at a critical parameter $\alpha_c\in(45,50)$, giving rise to a $T+$ state for $\alpha > \alpha_c$. 
These saddle points at $\alpha=50$ are illustrated in \Cref{fig:nlc}. 
When $\alpha$ slightly exceeds $\alpha_c$, the Hessian at the $T$ state exhibits a small-magnitude negative eigenvalue, while the Hessian at the $T+$ state has a small positive eigenvalue, as shown in \Cref{tab:ev}. 
Because of this pair of small-magnitude Hessian eigenvalues, the downward search from the $T$ state to the $T+$ state by 3-HiSD is severely slowed down.
As a result, small-magnitude eigenvalues persist throughout this search, leading to a large number of iterations.

\begin{table}
\centering
\caption{Smallest eigenvalues of Hessians at some states for $\alpha=45$ and $50$. 
Boldface numbers denote eigenvalues associated with bifurcations. }
\begin{tabular}{cccrrrrrr}
\toprule
$\alpha$ & State & Index & \multicolumn{6}{c}{Smallest eigenvalues} \\ \midrule  
45 & $T $ & 3& $-15.860$& $-13.473$& $-7.352$& $\mathbf{1.203}$& $3.019$& $4.691$   \\
50 & $T $ & 4& $-17.930$& $-15.283$& $-9.077$& $\mathbf{-0.415}$& $2.771$& $4.645$  \\
50 & $T+$ & 3& $-17.087$& $-14.414$& $-8.124$& $\mathbf{0.801}$& $2.817$& $4.510$  \\
\hline
45 & $C+$ & 4& $-18.185$& $-15.057$& $\mathbf{-0.527}$& $\mathbf{-0.527}$& $4.650$& $4.650$ \\
50 & $C+$ & 2& $-19.500$& $-16.790$& $\mathbf{0.308}$& $\mathbf{0.308}$& $5.138$& $5.138$\\
50 & $M+$ & 3& $-19.732$& $-16.815$& $\mathbf{-0.681}$& $\mathbf{0.005}$& $4.429$& $4.835$\\
50 & $N+$ & 4& $-19.738$& $-16.811$& $\mathbf{-0.699}$& $\mathbf{-0.005}$& $4.373$& $4.896$ \\
\bottomrule
\end{tabular}
\label{tab:ev}
\end{table}

Near the $C+$ state, the system exhibits severe ill-conditioning, as the emergence of the $M+$ and $N+$ states introduces additional small-magnitude Hessian eigenvalues. 
In these ill-conditioned cases, HiSD typically suffers from slow convergence due to small-magnitude Hessian eigenvalues, which motivates the use of SCS-HiSD. 
To demonstrate the computational efficiency of the proposed method, we perform a series of numerical experiments involving both upward and downward searches of $C+\leftrightarrow M+$, $M+ \leftrightarrow N+$, and $T+ \leftrightarrow T$ at $\alpha = 50$. 
The starting point is obtained by perturbing the source saddle point by $10^{-3}$ along the normalized eigenvector corresponding to the smallest-magnitude negative (resp. positive) Hessian eigenvalue for downward (resp. upward) searches.
The termination criterion is set as $\|\nabla E\|_{L^2}<10^{-6}$.

All methods employ BB step sizes clipped to $[\beta_0/5,\, 5\beta_0]$, with initial step size $\beta_0$.
The parameters are set as follows, which are nearly optimal based on extensive tuning: 
for HiSD, $\beta_0 = 2.4\times10^{-4}$; 
for A-HiSD, $\gamma_{\mathrm{M}} = 0.99$ and $\beta_0 = 8\times10^{-4}$; 
for SCS-HiSD, $\eta = 16$ and $\beta_0 = 2.4\times10^{-4}$.
To address the small positive eigenvalues at some saddle points, we employ the extended SCS-HiSD$+l$ method with $l=6-k$ to search for index-$k$ saddle points. 
Note that each SCS-HiSD$+l$ iteration involves $l$ additional directions, resulting in a per-iteration computational cost approximately $(1+l/k)$ times that of SCS-HiSD.

As a representative severely ill-conditioned case, the numerical results for the saddle search $M+ \to N+$ with $k=4$ are illustrated in \Cref{fig:err_nlc}. 
Even with nearly optimal step sizes, both HiSD and A-HiSD exhibit severe stagnation and fail to escape efficiently from the $M+$ state. 
Specifically, HiSD requires more than 6,000,000 iterations to reach the termination criterion, while A-HiSD takes more than 140,000 iterations. 
SCS-HiSD alleviates this ill-conditioning effect and attains substantially faster convergence in far fewer iterations. 
Furthermore, the SCS-HiSD$+l$ method substantially accelerates convergence, achieving the fastest decrease in the gradient norm within 3,000 iterations. 
In terms of wall time, the SCS-HiSD$+l$ method achieves a $900\times$ acceleration compared with HiSD, and $30\times$ compared with A-HiSD. 

\begin{figure}[htbp]
\centering
\includegraphics[width=0.5\linewidth]{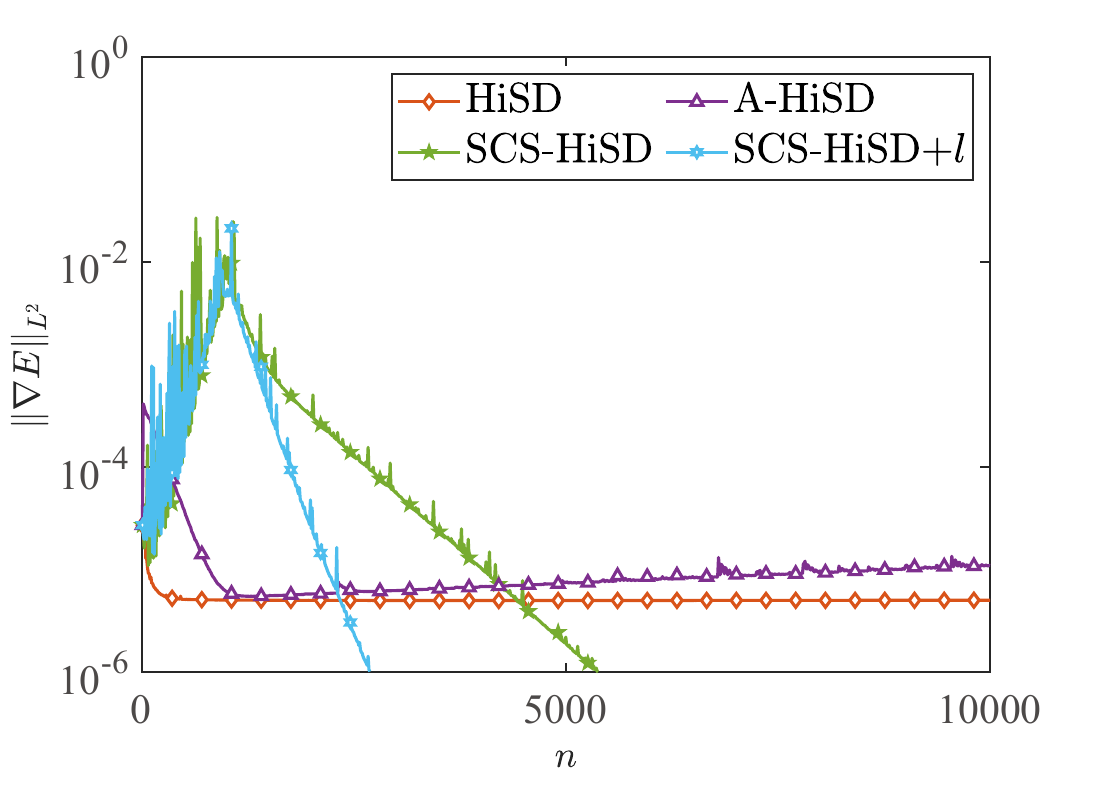}
\caption{\label{fig:err_nlc} 
Evolution of the gradient norm $\|\nabla E\|_{L^2}$ with respect to the iteration number $n$ for the saddle search $M+ \to N+$ in the liquid-crystal example. }
\end{figure}

\begin{table}[htbp]
\centering
\caption{\label{tab:iteration_comparison}
Wall time and number of iterations for saddle searches in the liquid-crystal example. ($l=6-k$)}
\begin{tabular}{cccccc}
\toprule
\multicolumn{2}{c}{Saddle search} & 
\multirow{2}{*}{$k$} & 
\multicolumn{3}{c}{Wall time/s (Number of iterations)} \\
\cmidrule(lr){1-2}\cmidrule(lr){4-6}
From & To & & \multicolumn{1}{c}{HiSD} & \multicolumn{1}{c}{A-HiSD} & {SCS-HiSD$+l$} \\
\midrule
$T+ $ & $T$  & 4 & $1478.5\,(84{,}366)$ & $47.4\,(2{,}960)$ & $69.6\,(2{,}346)$ \\
$T  $ & $T+$ & 3 & $9612.2\,(865{,}572)$ & $34.6\,(3{,}366)$ & $71.0\,(2{,}390)$ \\   
\midrule
$M+ $ & $C+$ & 2 & $558.8\,(113{,}421)$ & $13.7\,(3{,}006)$ & $73.0\,(2{,}413)$ \\
$C+ $ & $M+$ & 3 & $37101.8\,(3{,}310{,}293)$ & $832.3\,(76{,}870)$ & $90.6\,(3{,}067)$ \\
$N+ $ & $M+$ & 3 & $43399.7\,(6{,}015{,}628)$ & $1546.8\,(142{,}746)$ & $79.3\,(2{,}582)$ \\
$M+ $ & $N+$ & 4 & $68688.1\,(6{,}021{,}049)$ & $2445.9\,(142{,}751)$ & $72.1\,(2{,}687)$ \\
\bottomrule
\end{tabular}
\end{table}

\Cref{tab:iteration_comparison} summarizes the numerical results for various saddle searches. 
HiSD incurs a large computational cost, with the number of iterations increasing sharply as the smallest-magnitude Hessian eigenvalue approaches zero, as evidenced by \Cref{tab:ev}.
A-HiSD provides effective acceleration for locating the $T$, $T+$, and $C+$ states, where the relevant small-magnitude eigenvalues are on the order of $10^{-1}$. 
However, for severely ill-conditioned saddle points $M+$ and $N+$, the acceleration effect of A-HiSD remains insufficient, and these cases dominate the overall computational cost.
In contrast, the proposed methods effectively handle these ill-conditioned cases and consistently converge within 3,000 iterations for each saddle search.
This is consistent with the observations in the previous example, confirming that SCS-HiSD is particularly advantageous in locating severely ill-conditioned saddle points.

\section{Conclusions}\label{sec:6}
In this paper, we have proposed the SCS-HiSD method to accelerate HiSD for locating ill-conditioned saddle points.
By exploiting the approximate eigenvectors already maintained during the iterations, SCS-HiSD constructs a subspace inverse-Hessian approximation at negligible additional cost and incorporates it as an adaptive scaling along the unstable directions.
We rigorously establish the linear stability of the continuous SCS-HiSD system and provide a local convergence analysis for the discrete iterative scheme.
Our theoretical results show that, with either exact or inexact eigenvectors, the SCS-HiSD scheme achieves a substantially improved convergence rate in ill-conditioned cases.
Moreover, SCS-HiSD has been extended to mitigate the analogous slowdown caused by small positive Hessian eigenvalues.
Numerical experiments on benchmark functions and the liquid-crystal model confirm that SCS-HiSD substantially reduces both the iteration count and the overall computational cost, with pronounced acceleration in locating severely ill-conditioned saddle points. 
These results demonstrate that SCS-HiSD is a reliable and computationally efficient method for exploring complex energy landscapes, particularly for ill-conditioned saddle points.

In the present work, SCS-HiSD applies curvature scaling only on the unstable subspace, targeting ill-conditioning caused by small-magnitude negative Hessian eigenvalues.
Although the extended SCS-HiSD scheme can handle a small number of additional small positive eigenvalues, the acceleration may remain limited when many positive eigenvalues are close to zero.
A natural extension is to apply a quasi-Newton approach to the stable subspace without additional eigenvector computations; the algorithmic design and theoretical analysis of such a scheme are left for future work.

\bibliographystyle{siamplain}
\bibliography{references}
\end{document}